\documentclass[a4paper, 12pt]{amsart}
\usepackage[margin=2cm]{geometry}

\usepackage{amsmath}
\usepackage{amssymb}
\usepackage{calc}
\usepackage{enumitem}
\usepackage{graphicx}
\usepackage{tikz}
\usepackage{tikz-cd}
\usepackage{url}
\usepackage{xcolor}
\usepackage{etoolbox}
\usepackage{tabularx}
\usepackage{hyperref}
\usepackage{mathtools}
\usepackage{quiver}

\hypersetup{
  colorlinks   = true, 
  urlcolor     = {blue!50!black}, 
  linkcolor    = {blue!50!black}, 
  citecolor   = {red!50!black} 
}

\usepackage{cleveref}

\theoremstyle{plain}
\newtheorem{thm}{Theorem}[section]

\newtheorem{prop}[thm]{Proposition}
\newtheorem{lem}[thm]{Lemma}
\newtheorem{fact}[thm]{Fact}
\newtheorem*{thm*}{Theorem}
\newtheorem*{prop*}{Proposition}
\newtheorem*{lem*}{Lemma}
\newtheorem*{cor*}{Corollary}
\newtheorem*{claim*}{Claim}
\newtheorem*{fact*}{Fact}

\newtheorem{customthm}{Theorem}
\newenvironment{numthm}[1]
{\renewcommand\thecustomthm{#1}\customthm}
{\endcustomthm}

\theoremstyle{definition}
\newtheorem{defn}[thm]{Definition}

\theoremstyle{remark}

\newtheorem*{term*}{Terminology}

\newtheorem*{qn*}{Question}

\DeclareMathOperator{\Age}{Age}
\DeclareMathOperator{\ar}{ar}
\DeclareMathOperator{\Aut}{Aut}

\DeclareMathOperator{\dplus}{d_{_{+}}\!}

\DeclareMathOperator{\id}{id}

\DeclareMathOperator{\lvl}{l}
\DeclareMathOperator{\Lvl}{L}

\DeclareMathOperator{\Or}{Or}

\DeclareMathOperator{\scl}{scl}

\newcommand{\N}{
	\mathbb{N}
}

\newcommand{\mc}[1]{
	\mathcal{#1}
}
\newcommand{\Czero}{
	\mathcal{C}_0
}

\newcommand{\Dzero}{
	\mathcal{D}_0
}
\newcommand{\Cone}{
	\mathcal{C}_1
}
\newcommand{\Done}{
	\mathcal{D}_1
}
\newcommand{\tDone}{
    \tilde{\mathcal{D}}_1
}
\newcommand{\Oone}{
	\mathcal{O}_1
}
\newcommand{\CF}{
	\mathcal{C}_F
}

\newcommand{\ov}[1]{
    \overline{#1}
}

\newcommand{\sub}{
    \subseteq
}
\newcommand{\sq}{
    \sqsubseteq
}
\newcommand{\tri}{
    \trianglelefteq
}
\newcommand{\tld}[1]{
    \tilde{#1}
}

\newcommand{\Honza}{Hubi\v{c}ka }

\newcommand{\Jarik}{Ne\v{s}et\v{r}il }

\newcommand{\Fr}{Fra\"{i}ss\'{e} }

\expandafter\patchcmd\csname\string\proof\endcsname{\normalparindent}{0pt }{}{}

\title{Flows of linear orders on sparse graphs}
\author{Rob Sullivan}
\address{Rob Sullivan, Institut f\"{u}r Mathematische Logik, Universit\"{a}t  M\"{u}nster,  
Einsteinstraße 62,
48149  M\"{u}nster,
Germany}
\email{robertsullivan1990+maths@gmail.com}



\thanks{This project formed the second part of the PhD of the author at Imperial College London, under the supervision of Prof David Evans.}

\date{\today}

\begin{document}

    \begin{abstract}
        We consider the topological dynamics of the automorphism group of a particular sparse graph $M_1$ resulting from an ab initio Hrushovski construction. We show that minimal subflows of the flow of linear orders on $M_1$ have all orbits meagre, partially answering a question of Tsankov regarding results of Evans, \Honza and \Jarik on the topological dynamics of automorphism groups of sparse graphs.
    \end{abstract}
    
    \maketitle

\section{Introduction}

    The paper \cite{KPT05} of Kechris, Pestov and Todor\v{c}evi\'{c} established links between topological dynamics and structural Ramsey theory, with further developments in \cite{NVT13}, \cite{Zuc16}, \cite{BMT17} (among others). We assume the reader is familiar with the background here, and briefly recall three key results, which we formulate for \emph{strong} classes (classes of structures where we restrict to a particular subclass of permitted embeddings -- see \cite[Definition 2.1]{EHN19}):

    \begin{thm*}[{\cite[Theorems 1.1 \& 1.2, Corollary 3.3]{BMT17}}]
        Let $G$ be a Polish group with universal minimal flow $M(G)$.
        \begin{enumerate}
            \item $M(G)$ is metrisable iff $G$ has a coprecompact extremely amenable closed subgroup;
            \item if $M(G)$ is metrisable, then $M(G)$ has a comeagre orbit.
        \end{enumerate}
    \end{thm*}

    \begin{thm*}[{\cite[Theorem 4.8]{KPT05}}]
        Let $M$ be a \Fr limit of a strong amalgamation class $(\mc{K}, \leq)$. Then $\Aut(M)$ is extremely amenable iff $(\mc{K}, \leq)$ is a Ramsey class of rigid structures.
    \end{thm*}

    \begin{thm*}[{\cite[Theorem 10.8]{KPT05}, \cite[Theorem 5]{NVT13}, \cite[Theorem 5.7]{Zuc16}}]
        Let $M$ be the \Fr limit of an amalgamation class $(\mc{K}, \leq)$, and let $N$ be the \Fr limit of an amalgamation class $(\mc{K}^+, \leq^+)$ of rigid structures which is a reasonable strong expansion of $(\mc{K}, \leq)$. Let $G = \Aut(M), H = \Aut(N)$. Suppose $(\mc{K}^+, \leq^+)$ has the Ramsey property and the expansion property over $(\mc{K}, \leq)$, and suppose $H$ is a coprecompact subgroup of $G$.
			
        Then the universal minimal flow $M(G)$ of $G$ is metrisable and has a comeagre orbit. Explicitly, we have $M(G) = \widehat{G/H}$, the completion of the quotient $G/H$ of the right uniformity on $G$.
    \end{thm*}
    (We can also describe the comeagre orbit explicitly -- see the references for further details.)
    
    \medskip
    
    The paper \cite{EHN19}, which was the starting point for the current paper, showed that classes of sparse graphs used in Hrushovski constructions (\cite{Hru88}, \cite{Hru93}) demonstrate different behaviour to classes previously studied in the KPT context. A graph $A$ is \emph{$k$-sparse} if for all finite $B \sub A$, we have $|E(B)| \leq k |B|$. It is well-known (\cite{Nas64}, \cite[Theorem 3.4]{EHN19}) that a graph is $k$-sparse iff it is $k$-orientable. We take $k=2$ for presentational simplicity. 
    
    We briefly describe the classes $\Czero, \Cone, \CF$ of sparse graphs found in \cite{EHN18}, \cite{EHN19}. Let $\Czero$ denote the class of finite $2$-sparse graphs. For $A, B \in \Czero$, we write $A \leq_s B$ if there exists a $2$-orientation of $B$ in which $A$ is successor-closed (by \cite[Lemma 1.5]{Eva03}, this is equivalent to another phrasing in terms of predimension). With this notion of $\leq_s$-substructure, we have the free amalgamation class $(\Czero, \leq_s)$ with \Fr limit $M_0$ (this structure, an ``ab initio Hrushovski construction", was first described in \cite{Hru93}). 

    We also have a ``simplified" version of $M_0$, denoted by $M_1$ and first studied in \cite{Eva03}. This is the key structure we consider in this paper. Let $\Done$ denote the class of finite $2$-oriented graphs with no directed cycles, and let $\Cone$ be the class of graph reducts of structures in $\Done$. For $A, B \in \Cone$, write $A \leq_1 B$ if there exists an expansion $B^+ \in \Done$ in which $A$ is successor-closed. Then $(\Cone, \leq_1)$ is again a free amalgamation class, and we write $M_1$ for its \Fr limit.
    
    The structures $M_0, M_1$ are not $\omega$-categorical, but if we consider $2$-sparse graphs whose predimension is greater than a certain control function $F$ with logarithmic growth, and take another notion of $\leq_d$-substructure (where the predimension strictly increases), we obtain an $\omega$-categorical \Fr limit $M_F$ (see \cite{EHN19} and \cite[Section 3]{Eva13} for details). 
    
    We then have:

    \begin{thm*}[{\cite[Theorems 3.7, 3.16]{EHN19}}]
        Let $M = M_1, M_0, M_F$. Then $\Aut(M)$ has no coprecompact extremely amenable closed subgroup, and so its universal minimal flow is non-metrisable.

        Equivalently, by \cite[Theorem 4.8]{KPT05}, we have that $M$ has no coprecompact Ramsey expansion.
    \end{thm*}

    The case of $M_F$ is particularly interesting as it shows that the automorphism group of an $\omega$-categorical structure need not have ``tame" dynamics in the sense of metrisability of the universal minimal flow, and that $\omega$-categorical structures are not necessarily tame from a structural Ramsey theory perspective either.

    The paper \cite{EHN19} also investigates the existence of comeagre orbits. Let $\Or(M)$ denote the $\Aut(M)$-flow of $2$-orientations on $M$.

    \begin{thm*}[{\cite[Theorem 5.2]{EHN19}}]
        Let $M = M_1, M_0, M_F$. Let $Y$ be a minimal subflow of $\Or(M)$. Then all $\Aut(M)$-orbits of $Y$ are meagre.
    \end{thm*}

    Note that if $M(G)$ has a comeagre orbit, then so does any minimal $G$-flow (see \cite{AKL14}), so the above result shows again that $\Aut(M)$ for $M = M_1, M_0, M_F$ has non-metrisable universal minimal flow, using \cite[Theorem 1.2]{BMT17}. In the context of the above result, T.\ Tsankov asked the following (\cite[concluding remarks]{EHN19}):

    \begin{qn*}[Tsankov]
        Let $M = M_1, M_0, M_F$. Does $\Aut(M)$ have a (non-trivial) metrisable minimal flow with a comeagre orbit?
    \end{qn*}

    David Evans suggested that the author investigate the $\Aut(M)$-flow $\mc{LO}(M)$ of linear orders on $M$. We obtain the following result (the main result of this paper), for $M_1$, the ``simplified version" of $M_0$:
		
    \begin{numthm}{\ref{M1meagre}}
        Let $Y \sub \mc{LO}(M_1)$ be a minimal subflow. Then all $\Aut(M_1)$-orbits on $Y$ are meagre.
    \end{numthm}

    This result demonstrates that the phenomenon seen in \cite[Theorem 5.2]{EHN19} occurs more generally for other flows on $M_1$, partially answering the question of Tsankov. We also note in passing that $M_1$ is $\omega$-saturated and its theory is $\omega$-stable (see \cite{Eva03}).
    
    To prove \Cref{M1meagre}, we take the class of finite ordered graphs which $\leq$-embed into some element of the minimal flow $Y$, and show that this class fails to have the weak amalgamation property -- this gives the result, using \Cref{nowapmeagre}. To show failure of the weak amalgamation property, we will use the Ramsey expansion given by the \emph{admissible orders}, from \cite[Section 3.1]{EHN21}. We discuss these in \Cref{adm orders section}.
    
    The author has not been able to extend Theorem \ref{M1meagre} to $M_0$ and $M_F$, and believes that the proof strategy for $M_1$ would require significant modification for these cases. Partial results for $M_0$ (giving some information about minimal subflows of $\mc{LO}(M_0)$, and clarifying obstructions to the proof strategy) can be found in Chapter 5 of \cite{Sul22}, the author's PhD thesis.

    Recall that we consider $M_1$ to be a ``simplified version" of $M_0$. It would be interesting to know if there is an analogous ``simplified version" of $M_F$ -- if so, it may be possible to prove \Cref{M1meagre} for an $\omega$-categorical structure. See \cite[Chapter 7]{Sul22}.
 
\section{Background} \label{backgroundchap}
	
    We briefly summarise the material required for \Cref{adm orders section} and \Cref{M1meagre}. We assume thorough familiarity with \cite{EHN19} and the background and notation provided therein. This section contains no new material. (A reader looking for a less streamlined presentation may consult Chapter 1 of the author's PhD thesis \cite{Sul22}, again mostly based on \cite{EHN19}.)
    
    All first-order languages considered in this paper will be countable and relational, unless specified otherwise. (For the Ramsey result that we use, we will also need to consider countable languages consisting of relation symbols and set-valued function symbols, as in \cite{EHN21}.)
	
\subsection{Graphs and oriented graphs: notation} \label{graphnotationsec}
	
    We write $E_A \sub A^2$ for the (symmetrised) edge set of a graph $A$ (where $(x, y) \in E_A$ iff $\{x, y\}$ is an undirected edge of $A$), and write $\rho_A \sub A^2$ for an orientation of $E_A$ (see \cite[Definition 3.3]{EHN19}). A \emph{subgraph} will be a first-order substructure, i.e.\ an induced subgraph. For an oriented graph $(A, \rho_A)$, if $(x, y) \in \rho_A$, we write $xy \in \rho_A$, and for $x \in A$ we write $\dplus(x)$ for the out-degree of $x$.

\subsection{Sparse graphs: \texorpdfstring{$\Czero$}{C0} and \texorpdfstring{$\Cone$}{C1}}

    Recall (\cite[Definition 3.12]{EHN19}) that $\Czero$, $\Dzero$ denote the classes of finite $2$-sparse graphs and finite $2$-oriented graphs, and that $\Czero$ is the class of graph reducts of $\Dzero$. Let $\mc{D}_1$ be the class of finite $2$-oriented graphs with no directed cycles. By a slight abuse of terminology, we call a $2$-oriented graph with no directed cycles an \emph{acyclic $2$-oriented graph}. Let $\Cone$ be the class of graph reducts of $\Done$. We may consider $\Done$, $\Cone$ as ``simplified versions" of $\mc{D}_0$, $\mc{D}_1$: these classes are the \textbf{key examples} we are concerned with in this paper. They are originally from \cite{Eva03} and found in early preprint versions (\cite[Definition 3.16]{EHN18}) of \cite{EHN19}, though they do not appear in the published version. 
    
    We may also define $\Cone$ directly: the class $\Cone$ consists of the finite graphs $A$ where every non-empty subgraph $B \sub A$ has a vertex of degree $\leq 2$. This follows from the fact (\cite[Lemma 1.3]{Eva03}) that a finite graph $A$ has an acyclic $k$-orientation iff every non-empty subgraph $B$ has a vertex of degree $\leq k$ in $B$.

\subsection{Sparse graphs: \texorpdfstring{$\sq_s$}{successor-closedness} and \texorpdfstring{$\leq_1$}{1-closedness}}

    We now describe the distinguished notions of embedding used to define the particular strong classes $(\Cone, \leq_1)$, $(\Done, \sq_s)$. (Śee \cite[Definition 2.1]{EHN19} for strong classes and \cite[Section 3.4]{EHN19} for basic lemmas regarding $\sq_s$. The case of $\leq_1$ is originally from \cite{Eva05}.)

    For an oriented graph $A$ and $B \sub A$, we write $B \sq_s A$ to mean that $B$ is \emph{successor-closed} in $A$. For $B \sub A$, the \emph{successor-closure} $\scl(B)$ is the smallest successor-closed subset of $A$ containing $B$.

    Let $A, B \in \Cone$ with $A \sub B$. We write $A \leq_1 B$ if there exists an acyclic $2$-orientation $B^+ \in \Done$ of $B$ in which the induced orientation $A^+ \in \Done$ on $A$ has $A^+ \sq_s B^+$. By an argument entirely analogous to \cite[Section 3.4 and Lemma 4.8]{EHN19}, it is easy to show the following.

    \begin{fact*}[{\cite[Theorem 3.17]{EHN18}}] 
        $(\Cone, \leq_1)$ and $(\Done, \sq_s)$ are strong classes with free amalgamation, and the class $(\Done, \sq_s)$ is both a strong and a reasonable expansion of $(\Cone, \leq_1)$.
    \end{fact*}

    (Recall the definition of a strong expansion and a reasonable expansion from \cite[Definition 2.9, Definition 2.14]{EHN19}.)

    Let $M_1$ be the \Fr limit of $(\Cone, \leq_1)$ and let $G_1 = \Aut(M_1)$. We now discuss the technical framework used to analyse minimal subflows of $\mc{LO}(M_1)$, the $G_1$-flow of linear orders on $M_1$.

\subsection{The order expansion of an amalgamation class}

    Let $(\mc{K}, \leq)$ be an amalgamation class of $L$-structures, and let $L^\prec = L \cup \{\prec\}$ with $\prec$ binary. Let $\mc{K}^\prec$ be the class of $L^+$-structures $(A, \prec_A)$, where $A \in \mc{K}$ and $\prec_A$ is a linear order on $A$. For $(A, \prec_A), (B, \prec_B) \in \mc{K}^\prec$, write $(A, \prec_A) \leq (B, \prec_B)$ if $A \leq B$ and $\prec_A$ is the restriction of $\prec_B$ to $A$. We call $(\mc{K}^\prec, \leq)$ the \emph{order expansion} of $(\mc{K}, \leq)$. It is straightforward to check that $(\mc{K}^\prec, \leq)$ is a strong class which is both a strong and reasonable expansion of $(\mc{K}, \leq)$.

\subsection{Flows from reasonable expansions}

    Let $L \sub L^+$ be relational languages. Let $(\mc{K}, \leq)$ be an amalgamation class of $L$-structures with \Fr limit $M$, and let $\mc{D}$ be a reasonable $L^+$-expansion of $(\mc{K}, \leq)$. Recall (\cite[Section 2.3, Theorem 2.15]{EHN19}) that we obtain an $\Aut(M)$-flow $X(\mc{D})$ from $\mc{D}$ by taking $X(\mc{D})$ to be the set consisting of the $L^+$-expansions $M^+$ of $M$ such that $M^+|_A \in \mc{D}$ for all $A \leq M$. (here, $M^+|_A$ denotes the $L^+$-structure induced on the domain of $A$ by $M^+$), where the topology on $X(\mc{D})$ is given by: for $B \leq M$ with expansion $B^+ \in \mc{D}$, we specify a basic open set $U(B^+) = \{M^+ \in X(\mc{D}) : M^+|_B = B^+\}$.
	
    As the order expansion $\mc{K}^\prec$ of an amalgamation class $(\mc{K}, \leq)$ (with \Fr limit $M$) is reasonable, we have that $X(\mc{K}^\prec)$ is an $\Aut(M)$-flow: we denote this by $\mc{LO}(M)$, the \emph{flow of linear orders} on $M$. As $\Done$ is a reasonable expansion of $(\Cone, \leq_1)$, we have that $X(\Done)$ is a $G_1$-flow, which we denote by $\Or(M_1)$, the \emph{flow of orientations}.

    We will need a very mild reformulation of \cite[Lemma 2.16]{EHN19}, a technical result regarding subflows of $X(\mc{D})$:
    
    \begin{fact} \label{YfindD'}
        Let $\mc{D}$ be a reasonable expansion of an amalgamation class $(\mc{K}, \leq)$ with \Fr limit $M$. Let $Y$ be a subflow of $X(\mc{D})$. Let $\mc{D}' \sub \mc{D}$ be the class of finite $L^+$-structures which $\leq$-embed into some element of $Y$. Then:
        \begin{enumerate}
            \item $\mc{D}'$ is a reasonable expansion of $(\mc{K}, \leq)$ with $X(\mc{D}') = Y$;
            \item if $Y = \ov{G \cdot M_0^+}$ for some $M_0^+ \in X(\mc{D})$, then $\mc{D}' = \Age_\leq(M_0^+)$.
        \end{enumerate}
    \end{fact}
    \begin{proof}
        (1): \cite[Lemma 2.16]{EHN19}. (2): Let $A^+ \in \mc{D}'$. We may assume $A^+ \leq M_1^+$ for some $M_1^+ \in Y$. As $Y = \ov{G \cdot M_0^+}$, there is $g \in \Aut(M)$ such that $A^+ = M_1^+|_A = (gM_0^+)|_A$, so $A^+$ $\leq$-embeds into $M_0^+$ and thus $\mc{D}' \sub \Age_\leq(M_0^+)$. The reverse inclusion is immediate as $M_0^+ \in Y$.
    \end{proof}

\subsection{The expansion property}
        
    Let $(\mc{K}, \leq)$ be an amalgamation class with \Fr limit $M$ and let $\mc{D}$ be a reasonable expansion of $(\mc{K}, \leq)$. Recall that $\mc{D}$ has the \emph{expansion property} over $(\mc{K}, \leq)$ if for $A \in \mc{K}$, there exists $B$ in $\mc{K}$ with $A \leq B$ such that for all expansions $A^+, B^+$ of $A, B$ in $\mc{D}$, there exists a $\leq$-embedding $A^+ \to B^+$. Also recall the following fact:
	
    \begin{fact}[{\cite[Theorem 2.18]{EHN19}}] \label{miniffexp}
        The $G$-flow $X(\mc{D})$ is minimal iff $\mc{D}$ has the expansion property over $(\mc{K}, \leq)$.
    \end{fact}
	
\subsection{Meagre orbits} \label{meagrebackground}
	
    We recall the weak amalgamation property (WAP), which will be crucial in the proof of Theorem \ref{M1meagre}. (WAP was first defined in \cite{KR07}, \cite{Iva99}. See \cite[Section 2.4]{EHN19} for the formulation of WAP for strong classes.)
	
    Let $\mc{D}$ be a reasonable class of $L^+$-expansions of an $L$-amalgamation class $(\mc{K}, \leq)$. Recall that $(\mc{D}, \leq)$ has the \emph{weak amalgamation property} (WAP) if for all $A \in \mc{D}$, there exists $B \in \mc{D}$ and a $\leq$-strong $L^+$-embedding $f : A \to B$ such that, for any $\leq$-strong $L^+$-embeddings $f_i : B \to C_i \in \mc{D}$ ($i = 0, 1$), there exists $D \in \mc{D}$ and $\leq$-strong $L^+$-embeddings $g_i : C_i \to D$ ($i = 0, 1$) with $g_0 \circ f_0 \circ f = g_1 \circ f_1 \circ f$. (Note that here we specify only that the diagram commutes for $A$.)

    \begin{fact}[{\cite[Lemma 2.23]{EHN19}}] \label{nowapmeagre}
        Let $\mc{D}$ be a reasonable class of $L^+$-expansions of an $L$-amalgamation class $(\mc{K}, \leq)$ with \Fr limit $M$. Suppose that $X(\mc{D})$ is a minimal flow.
			
        If $(\mc{D}, \leq)$ does not have the weak amalgamation property, then all $\Aut(M)$-orbits on $X(\mc{D})$ are meagre.
    \end{fact}
    For a proof, see \cite[Lemma 1.77]{Sul22} (a straightforward correction of the proof in \cite{EHN19}).
	
\section{Admissible orders: a Ramsey expansion of \texorpdfstring{$(\Done, \sq_s)$}{D1 with successor-closed substructures}} \label{adm orders section}

    We now provide an explicit description of a Ramsey expansion of $(\Done, \sq_s)$, given by the \emph{admissible orders} on $(\Done, \sq_s)$, using Theorem 1.4 of \cite{EHN21}. This will be an essential tool in the proof of \Cref{M1meagre}. (This Ramsey expansion will also have the expansion property over $(\Done, \sq_s)$, though we will not use this.) In the below two definitions, we adapt definitions from \cite[Section 1 and Section 3]{EHN21} to the specific case of $(\Done, \sq_s)$.

    \begin{defn} \label{finlvldef}
        Let $A \in \Done$. For $a \in A$, let $a^\circ = \scl_A(a) \setminus \{a\}$.
        
        For $a \in A$, we inductively define the \emph{level} $\lvl_A(a)$ of $a$ as follows. If $\scl_A(a) = \{a\}$, then $\lvl_A(a) = 0$. Otherwise, let $b$ be a vertex of $a^\circ$ of maximum level, and then define $\lvl_A(a) = \lvl_A(b) + 1$. We write $\Lvl_n(A)$ ($n \geq 0$) for the set of vertices of $A$ of level $n$.

        We say that $a, b \in A$ are \emph{homologous} if $a^\circ = b^\circ$ and there is an isomorphism $\scl_A(a) \to \scl_A(b)$ which is the identity on $a^\circ = b^\circ$. We let $Q_A(a)$ denote the set of vertices of $A$ homologous to $a$, and call $Q_A(a)$ the \emph{cone} of $a$.

        If there is $a \in A$ with $A = \scl_A(a)$, we call $A$ a \emph{closure-extension} with \emph{head vertex} $a$, and write $A^\circ = a^\circ$. (Note that $a$ is necessarily unique.)
    \end{defn}

    \begin{defn} \label{adm order def}
        Fix a linear order $\tri$ on the set of isomorphism types of ordered closure-extensions $A^\prec$ such that:
        \begin{itemize}
            \item[($\ast$)] if $|A| < |B|$, then $A^\prec \lhd B^\prec$.
        \end{itemize}

        We say that a class $\mc{O} \sub \Done^\prec$ is a \emph{class of admissible orderings} of structures in $\Done$ if:
        \begin{enumerate}
            \item each $A \in \Done$ has an expansion $A^\prec \in \mc{O}$;
            \item $\mc{O}$ is closed under $\sq_s$-substructures;
            \item for $A^\prec \in \mc{O}$ and $u, v \in A$, if:
            \begin{itemize}
                \item $\scl_A(u)^\prec \lhd \scl_A(v)^\prec$, or
                \item $\scl_A(u)^\prec \cong \scl_A(v)^\prec$ and $u^\circ$ is lexicographically before $v^\circ$ in the order $\prec_A$,
            \end{itemize}
                then $u \prec_A v$;
            \item for each $B \in \Done$, if $A_1, \cdots, A_n \sq_s B$ and $\prec'$ is a linear order on $A = \bigcup_{i \leq n} A_i$ such that $\prec'$ satisfies (3) and each $A_i$ is admissibly ordered by $\prec'$, then there exists an admissible order $\prec_B$ on $B$ extending $\prec'$;
        \end{enumerate}
    \end{defn}

    (In the above, we adapt \cite[Definition 3.5]{EHN21}. Several aspects of the general definition in \cite{EHN21} simplify in this case: (A4) can be omitted as closure components are single vertices, and (A6) follows from $(\ast)$ and $(3)$.)

    The below theorem is an immediate translation of \cite[Theorem 1.4]{EHN21} to the context of this paper. We will explain how to adapt \cite[Theorem 1.4]{EHN21} to our context at the end of this section.

    \begin{prop} \label{Oone exists}
        There exists a class $\Oone \sub \Done^\prec$ of admissible orderings of structures in $\Done$. We have that $(\Oone, \sq_s)$ is an amalgamation class, and $(\Oone, \sq_s)$ has the Ramsey property and the expansion property over $(\Done, \sq_s)$.
    \end{prop}

    \begin{lem}
        $(\Oone, \sq_s)$ is a strong expansion of $(\Done, \sq_s)$.
    \end{lem}
    \begin{proof}
        Parts (1) and (2) in the definition of strong expansion are immediate. Part (3) follows immediately from part (4) of the definition of admissible orders. 
    \end{proof}

    We now give a specific property resulting from \Cref{adm order def} that we will use in the proof of \Cref{M1meagre}, the main result of this paper.
    \begin{lem} \label{vertex greater than base in adm order}
        Let $A^\prec \in \Oone$. Let $a \in A$ and let $b \in a^\circ$. Then $b \prec_A a$.
    \end{lem}
    \begin{proof}
        As $|\scl_A(b)| < |\scl_A(a)|$, by parts $(\ast)$ and $(3)$ in \Cref{adm order def} we have $b \prec_A a$.
    \end{proof}

    We now explain how to adapt \cite[Theorem 1.4]{EHN21} and the definition of admissible orders found in \cite{EHN21} to give the definitions and theorem above. The paper \cite{EHN21} gives Ramsey expansions for classes of finite structures in languages that may include \emph{set-valued function symbols}, which enables us to deal with the strong class $(\Done, \sq_s)$.

    \begin{defn}[{\cite[Section 1]{EHN21}}]
        A language $L = L_R \cup L_F$ of relation and set-valued function symbols consists of a set $L_R$ of relation symbols and a set $L_F$ of \emph{set-valued} function symbols $L_F$, where each symbol has an associated arity $n \in \N_+$. 
        
        An $L$-structure $(A, (R_A)_{R \in L_R}, (F_A)_{F \in L_F})$ consists of a set $A$ (the domain) together with sets $R_A \sub A^n$ for each relation symbol $R \in L_R$ of arity $n$ and functions $F_A : A^n \to \mc{P}(A)$ for each set-valued function symbol $F \in L_F$ of arity $n$. Usually we will just write $A$ to denote the structure.

        A function $f : A \to B$ between $L$-structures $A, B$ is an \emph{embedding} if $f$ is injective and:
        \begin{itemize}
            \item for each relation symbol $R \in L_R$ of arity $n$,
            \[(a_1, \cdots, a_n) \in R_A \Leftrightarrow (f(a_1), \cdots, f(a_n)) \in R_B;\]
            \item for each set-valued function symbol $F \in L_F$ of arity $n$,
            \[f(F_A(a_1, \cdots, a_n)) = F_B(f(a_1), \cdots, f(a_n)).\]
        \end{itemize}

        For $L$-structures $A, B$, we say that $A$ is a substructure of $B$, written $A \sub B$, if the domain of $A$ is a subset of the domain of $B$ and the inclusion map $A \hookrightarrow B$ is an embedding of $L$-structures.

        We define the hereditary property, joint embedding property, amalgamation property, Ramsey property and expansion property for classes of $L$-structures exactly as for usual first-order languages, and we also define amalgamation classes and Ramsey classes as before.
        
        Let $A \sub B_0, B_1$ be $L$-structures, and suppose that $B_0 \cap B_1 = A$. The \emph{free amalgam} of $B_0, B_1$ over $A$ is the $L$-structure $C$ with domain $B_0 \cup B_1$, where $R_C = R_{B_0} \cup R_{B_1}$ for each $R \in L_R$ and where, for each $F \in L_F$ of arity $n$, the function $F_C : C^n \to \mc{P}(C)$ is defined by $F_C(\bar{c}) = F_{B_i}(\bar{c})$ for $\bar{c} \in B_i^n$ ($i = 0, 1$) and $F_C(\bar{c}) = \varnothing$ otherwise. An amalgamation class where amalgams can always be taken to be free amalgams is called a \emph{free amalgamation class}.
    \end{defn}

    The above framework enables us to deal with $(\Done, \sq_s)$ as follows (\cite[Section 5.1]{EHN21}). Let $\tld{L}$ consist of the binary relational language $L_\text{or}$ of oriented graphs together with a unary set-valued function symbol $F$. Let $\tDone$ consist of the $\tld{L}$ structures $\tld{A} = (A, F_A)$ where $A \in \Done$ and $F_A : A \to \mc{P}(A)$ is a function sending each vertex of $A$ to its out-neighbourhood in $A$. Then there is a bijection $\Done \to \tDone$ sending each $A \in \Done$ to its unique $\tld{L}$-expansion $\tld{A}$ in $\tDone$, and for $A, B \in \Done$, we have that $A \sq_s B$ iff $\tld{A} \sub \tld{B}$. We then have that $\tld{L}$-embeddings between elements of $\tDone$ are $\sq_s$-embeddings when considered in the language $L_\text{or}$, and therefore $\tDone$ is a free amalgamation class.

    We now recall \cite[Theorem 1.4]{EHN21}, which will give us an explicitly defined Ramsey expansion of $\tDone$ via admissibly ordered structures.

    \begin{thm*}[{\cite[Theorem 1.4]{EHN21}}]
        Let $L$ be a language (consisting of relation and set-valued function symbols). Let $\mc{K}$ be a free amalgamation class of $L$-structures. Then there exists an explicitly defined amalgamation class $\mc{O} \sub \mc{K}^\prec$ of \emph{admissible orderings} such that:
        \begin{itemize}
            \item every $A \in \mc{K}$ has an ordering in $\mc{O}$;
            \item the class $\mc{O}$ has the Ramsey property and the expansion property over $\mc{K}$.
        \end{itemize}
    \end{thm*}

    The above theorem, together with \cite[Definition 3.5]{EHN21}, which gives the explicit definition of admissible orders, gives a Ramsey expansion $\tld{\mc{O}}_1$ of $\tDone$. Using the correspondence between $\tDone$ and $(\Done, \sq_s)$ detailed in the preceding paragraph, we thus obtain a Ramsey expansion $(\Oone, \sq_s)$ of $(\Done, \sq_s)$ satisfying the conditions of \Cref{adm order def} (this definition is just a direct adaptation of \cite[Definition 3.5]{EHN21}).
 
\section{Minimal subflows of \texorpdfstring{$M_1$}{M1}} \label{M1section}

    In this section, we prove the main theorem of this paper:

    \begin{thm} \label{M1meagre}
        Let $Y \sub \mc{LO}(M_1)$ be a minimal subflow of $\mc{LO}(M_1)$. Then all $G$-orbits on $Y$ are meagre.
    \end{thm}

\subsection{Preparatory definitions and lemmas}

    \begin{defn}
        Let $N_1 = (M_1, \rho)$ be the \Fr limit of $(\Done, \sq_s)$, and let $(N_1, \prec_\alpha)$ be the \Fr limit of $(\Oone, \sq_s)$. (Here we use \cite[Theorem 2.10]{EHN19}: $(\Done, \sq_s)$ is a strong expansion of $(\Cone, \leq_1)$ and $(\Oone, \sq_s)$ is a strong expansion of $(\Done, \sq_s)$.)
        
        Recall that $G_1 = \Aut(M_1)$. Let $H_1 = \Aut(N_1, \alpha)$. As $(\Oone, \sq_s)$ has the Ramsey property by \Cref{Oone exists}, by the fundamental result of the KPT correspondence (\cite[Theorem 4.8]{KPT05}) formulated for strong classes (\cite[Theorem 2.13]{EHN19}) we have that $H_1$ is extremely amenable.
			
        We will write $G = G_1, H = H_1$ in the remainder of \Cref{M1section} for ease of notation.
    \end{defn}
		
    \begin{defn}
        Let $a \in N_1$. As $N_1$ is the union of an increasing chain of $\sq_s$-substructures, we have that $\scl_{N_1}(a)$ is finite, and for any $A \sq_s N_1$ we have $\scl_{N_1}(a) = \scl_A(a)$. We define $a^\circ = \scl_{N_1}(a) \setminus \{a\}$, and define homologous vertices and cones in $N_1$ as in \Cref{finlvldef}. We define the \emph{level} $\lvl_{N_1}(a)$ of $a$ in $N_1$, usually just denoted $\lvl(a)$, to be the level of $a$ in $\scl_{N_1}(a)$.
    \end{defn}
		
    \begin{lem} \label{betaoncones}
        Let $\prec_\beta$ be an $H$-fixed point in the flow $H \curvearrowright \mc{LO}(M_1)$, and let $Q$ be a cone of $N_1$. Then $\prec_\beta$ agrees with either $\prec_\alpha^{}$ or $\prec_\alpha'$ on $Q$, where $\prec_\alpha'$ denotes the reverse of the linear order $\prec_\alpha$.
    \end{lem}
    \begin{proof}
        Take $a_0, b_0 \in Q$ with $a_0 \prec_\alpha b_0$. Then for $a, b \in Q$ with $a \prec_\alpha b$, by \Cref{vertex greater than base in adm order} there exists an ordered digraph isomorphism $f : \scl_{N_1}(a_0, b_0)^{\prec_\alpha} \to \scl_{N_1}(a, b)^{\prec_\alpha}$ with $f(a_0) = a, f(b_0) = b$, and by $\sq_s$-ultrahomogeneity we may extend to an element $f \in H$.
			
        As $H \sub G_\beta$, $f$ is $\beta$-preserving. If $a_0 \prec_\beta b_0$, then $f(a_0) \prec_\beta f(b_0)$, so $a \prec_\beta b$, and so $\prec_\beta$ agrees with $\prec_\alpha$ on $Q$. If $a_0 \succ_\beta b_0$, then $\prec_\beta$ agrees with $\prec_\alpha'$ on $Q$.
    \end{proof}

\subsection{Setup and proof notation}
    Before beginning the proof, we first need to set up our approach.
		
    Let $Y \sub \mc{LO}(M_1)$ be a minimal subflow of $G \curvearrowright \mc{LO}(M_1)$. As $H$ is extremely amenable, the flow $H \curvearrowright Y$ has an $H$-fixed point $\prec_\beta$, and as $Y$ is a minimal $G$-flow, we have $Y = \ov{G \,\cdot\!\prec_\beta}$. Let $\mc{J} = \Age_{\leq_1}(M_1, \prec_\beta)$. By \Cref{YfindD'}, we have $Y = X(\mc{J})$. We will show that $(\mc{J}, \leq_1)$ does not have the weak amalgamation property (WAP), which implies that all $G$-orbits on $Y$ are meagre by \Cref{nowapmeagre}.
    
    We will now use the above notation throughout the rest of this section.
		
\subsection{Proof idea - informal overview}
		
    We will assume $(\mc{J}, \leq_1)$ has WAP, for a contradiction. Let $\{a_0\} \in \mc{J}$ be a singleton with the trivial linear order. By assumption $\{a_0\}$ has a WAP-witness $A^\prec$. We will then construct $\leq_1$-embeddings of $A^\prec$ into two ordered graphs $C_0^\prec, C_1^\prec \in \mc{J}$ which are WAP-incompatible: it will not be possible to find $D^\prec$ completing the WAP commutative diagram for $\{a_0\}$ with the two embeddings, and this will give a contradiction.
		
    The incompatibility of the two ordered graphs $C_0^\prec, C_1^\prec$ in $\mc{J}$ will result from them forcing incompatible orientations: we can use the order $\prec_\beta$ to force certain edge orientations in $\rho$. The incompatible orientations will consist of a binary out-directed tree $T_0$ and a binary out-directed tree with the successor-closures of two vertices identified, which we denote by $T_1$: these cannot start from the same point of a $2$-orientation, as one contains a $4$-cycle and the other does not. The idea to use two incompatible orientations in the WAP commutative diagram comes from the proof of \cite[Theorem 5.2]{EHN19}.
		
    The key difficulties in the proof of Theorem \ref{M1meagre} are showing that we can use $\prec_\beta$ (specifically, particular finite ordered graphs in $\mc{J} = \Age_{\leq_1}(M_1, \prec_\beta)$) to force orientations of edges in $\rho$ (Lemma \ref{digraphWOG}), and also showing that the ordered graphs that we construct to force orientations of edges in $\rho$ do in fact lie in $\mc{J}$ (Lemma \ref{existsWOG}).
		
\subsection{Attaching trees and near-trees} 

    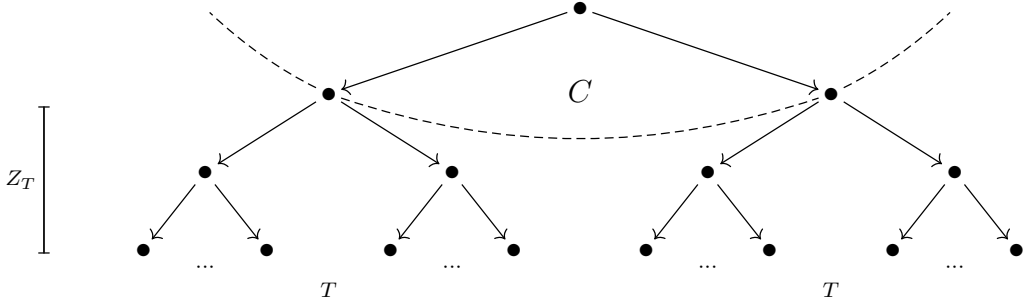
\begin{figure}
        \centering
        \[\begin{tikzcd}[cramped, , sep=small, nodes={inner sep=1.5pt}]
            &&&& {} &&&&&& \bullet &&&&&& {} &&& \\ \\
            & {} &&&&& \bullet &&&& C &&&& \bullet &&& \\ \\
            && {} && \bullet && {} && \bullet &&&& \bullet &&&& \bullet &&& \\ \\
            {} & {} && \bullet && \bullet && \bullet && \bullet && \bullet && \bullet && \bullet && \bullet &&&
            \arrow[from=1-11, to=3-7]
            \arrow[from=1-11, to=3-15]
            \arrow["{Z_T}"', maps to, no head, from=3-2, to=7-2]
            \arrow[maps to, no head, from=7-2, to=3-2]
            \arrow[curve={height=-3pt}, dashed, no head, from=3-7, to=1-5]
            \arrow[curve={height=22pt}, dashed, no head, from=3-7, to=3-15]
            \arrow[from=3-7, to=5-5]
            \arrow[from=3-7, to=5-9]
            \arrow[curve={height=3pt}, dashed, no head, from=3-15, to=1-17]
            \arrow[from=3-15, to=5-13]
            \arrow[from=3-15, to=5-17]
            \arrow[from=5-5, to=7-4]
            \arrow[from=5-5, to=7-6]
            \arrow[from=5-9, to=7-8]
            \arrow[from=5-9, to=7-10]
            \arrow[from=5-13, to=7-12]
            \arrow[from=5-13, to=7-14]
            \arrow[from=5-17, to=7-16]
            \arrow[from=5-17, to=7-18]
            \arrow["T"', curve={height=12pt}, draw=none, from=7-4, to=7-10]
            \arrow["T"', curve={height=12pt}, draw=none, from=7-12, to=7-18]
            \arrow["\cdots"', curve={height=2pt}, draw=none, from=7-4, to=7-6]
            \arrow["\cdots"', curve={height=2pt}, draw=none, from=7-8, to=7-10]
            \arrow["\cdots"', curve={height=2pt}, draw=none, from=7-12, to=7-14]
            \arrow["\cdots"', curve={height=2pt}, draw=none, from=7-16, to=7-18]
        \end{tikzcd}\]
        \caption{The oriented graph $D_T$.}
        \label{attachtrees}
    \end{figure}
    
    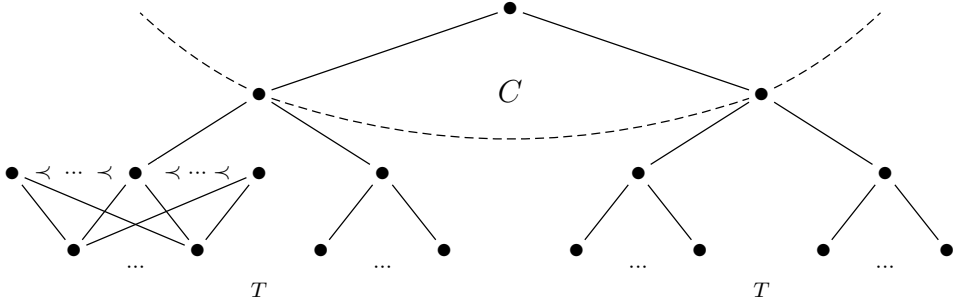
\begin{figure}
    \centering
        \[\begin{tikzcd}[cramped, sep=small, nodes={inner sep=1.5pt}]
            &&&& {} &&&&&& \bullet &&&&&& {} &&& \\ \\
            & {} &&&&& \bullet &&&& C &&&& \bullet &&& \\ \\
            && {\bullet } && \bullet && \bullet && \bullet &&&& \bullet &&&& \bullet &&& \\ \\
            {} & {} && \bullet && \bullet && \bullet && \bullet && \bullet && \bullet && \bullet && \bullet &&&
            \arrow[no head, from=1-11, to=3-7]
            \arrow[no head, from=1-11, to=3-15]
            \arrow[curve={height=-3pt}, dashed, no head, from=3-7, to=1-5]
            \arrow[curve={height=22pt}, dashed, no head, from=3-7, to=3-15]
            \arrow[no head, from=3-7, to=5-5]
            \arrow[no head, from=3-7, to=5-9]
            \arrow[curve={height=3pt}, dashed, no head, from=3-15, to=1-17]
            \arrow[no head, from=3-15, to=5-13]
            \arrow[no head, from=3-15, to=5-17]
            \arrow["{\prec \;\; \cdots \;\; \prec}"{description}, draw=none, from=5-3, to=5-5]
            \arrow[no head, from=5-3, to=7-4]
            \arrow[no head, from=5-3, to=7-6]
            \arrow["{\prec\; \cdots \;\prec}"{description}, draw=none, from=5-5, to=5-7]
            \arrow[no head, from=5-5, to=7-4]
            \arrow[no head, from=5-5, to=7-6]
            \arrow[no head, from=5-7, to=7-4]
            \arrow[no head, from=5-7, to=7-6]
            \arrow[no head, from=5-9, to=7-8]
            \arrow[no head, from=5-9, to=7-10]
            \arrow[no head, from=5-13, to=7-12]
            \arrow[no head, from=5-13, to=7-14]
            \arrow[no head, from=5-17, to=7-16]
            \arrow[no head, from=5-17, to=7-18]
            \arrow["T"', curve={height=12pt}, draw=none, from=7-4, to=7-10]
            \arrow["T"', curve={height=12pt}, draw=none, from=7-12, to=7-18]
            \arrow["\cdots"', curve={height=2pt}, draw=none, from=7-4, to=7-6]
            \arrow["\cdots"', curve={height=2pt}, draw=none, from=7-8, to=7-10]
            \arrow["\cdots"', curve={height=2pt}, draw=none, from=7-12, to=7-14]
            \arrow["\cdots"', curve={height=2pt}, draw=none, from=7-16, to=7-18]
        \end{tikzcd}\]
        \caption{The ordered graph $C_T^\prec$ with witness vertices indicated on one vertex.}
        \label{T-witness ordered graph}
    \end{figure}
        
    For $q \in \N_+$, let $T_0(q)$ be the digraph given by a binary tree of height $2q + 1$, oriented outwards towards the leaves and with head vertex $c$. Let $T_1(q)$ be the digraph given by taking $T_0(q)$ and identifying the successor-closures of two vertices at height $q+2$ whose paths to the head vertex $c$ meet at height $q$. We have $T_0(q), T_1(q) \in \mc{D}_1$. 
		
    Let $T$ be one of the digraphs $T_0(q)$ or $T_1(q)$. Take $C \in \mc{D}_1$ with each vertex having out-degree $2$ or $0$. Let $D_T$ be the digraph consisting of $C$ together with, for each vertex $v \in C$ with $\dplus(v) = 0$, a copy of $T$ attached at $v$, where we identify $c$ and $v$. Let $Z_T$ denote the sub-digraph of $D_T$ whose vertices are the vertices of the copies of $T$ attached to $C$ in $D_T$. Let $D_T^-$ denote the graph reduct of $D_T$. (We will use this notation throughout this section. See \Cref{attachtrees}.)
		
    We have $D_T \in \mc{D}_1$. Let ${D_T}'$ be the acyclic $2$-reorientation of $D_T$ where the copies of $T$ have been oriented so that the non-head vertices of each copy of $T$ are directed towards the head vertex $c$, leaving the orientation on vertices of $C$ unchanged. Then we have $C \sq_s {D_T}'$ in this reorientation, and so $C^- \leq_1 D_T^-$.
		
    \begin{defn} \label{defnWOG}
        Let $C \in \mc{D}_1$ with each vertex having out-degree $2$ or $0$, and let $D_T$ be defined as above.
			
        An ordered graph $C_T^\prec \in \mc{J}$ is a \emph{$T$-witness ordered graph for C} if:
        \begin{itemize}
            \item $C_T$ consists of the graph reduct ${D_T}^-$ of $D_T$ together with, for each non-leaf tree vertex $v$ of $D_T$, an additional $10$ copies of $\scl_{D_T}(v)$ freely amalgamated (as graphs) over $\scl_{D_T}(v)^\circ$, and $C \leq_1 C_T$;
            \item for each non-leaf tree vertex $v \in D_T$, the additional $10$ copies of $v$ may be labelled as $v_{-5}, \cdots, v_{-1}, v_1, \cdots, v_5$ so that $v_{-5} \prec \cdots \prec v_{-1} \prec v \prec v_1 \prec \cdots \prec v_5$ in $\prec_{C_T}$. (We call these $v_i$ the \emph{witness vertices} of $v$.)
        \end{itemize}
        (See \Cref{T-witness ordered graph}.)
    \end{defn}

    The following is the key lemma here.
	
    \begin{lem} \label{digraphWOG}
        Let $C \in \mc{D}_1$ with each vertex having out-degree $2$ or $0$, and let $C_T^\prec \in \mc{J}$ be a $T$-witness ordered graph for $C$. As $C_T^\prec \in \mc{J}$, there exists a $\leq_1$-ordered graph embedding $\theta : C_T^\prec \to (M_1^{}, \prec_\beta)$. Then, considering the digraph structure on $Z_T$ induced by $D_T$, $\theta|_{Z_T} : Z_T \to (M_1, \rho)$ is also a digraph embedding.
    \end{lem}
		
    \begin{proof}
        We may take $\theta = \id$ for ease of notation. Take $v, x, y \in Z_T$ with out-edges $vx, vy$ in the orientation of $Z_T$. We need to show that $v$ has out-edges $vx, vy$ in the orientation $\rho$ of $M_1$. Let $v_{-5}, \cdots, v_{-1}, v_1, \cdots, v_5$ be the witness vertices of $v$ in $C_T^\prec$, and let $v_0 = v$. As $\theta$ is a $\leq_1$-ordered graph embedding, we have that $v_i \prec_\beta v_j$ for $i < j$, and we have undirected edges $v_ix, v_iy$ for $-5 \leq i \leq 5$.
			
        As $\rho$ is a $2$-orientation, for some $i$ with $-5 \leq i \leq -1$ we must have that $v_ix, v_iy$ are out-edges of $\rho$, and likewise for some $j$ with $1 \leq j \leq 5$ we must have that $v_jx, v_jy$ are out-edges of $\rho$. If either $xv_0 \in \rho$ or $yv_0 \in \rho$, then $v_0 \in \scl_\rho(x, y)$, and as $v_i, v_j$ lie in the same cone, by \Cref{vertex greater than base in adm order} there exists $h \in H$ with $hv_i = v_j$ and $h$ fixing $v_0$. As $H \sub G_\beta$, we have that $h \in G_\beta$. But $v_i \prec_\beta v_0$, so $hv_i \prec_\beta hv_0$, thus $v_j \prec_\beta v_0$ - contradiction. So therefore $v_0x, v_0y \in \rho$.
    \end{proof}
		
    \begin{lem} \label{existsWOG}
        Let $C \in \mc{D}_1$ with each vertex having out-degree $2$ or $0$. Then there exists a $T$-witness ordered graph $C_T^\prec \in \mc{J}$ for $C$.
    \end{lem}
	
    \begin{proof}
        Let $d_1, \cdots, d_k$ be an enumeration of the non-leaf tree vertices of $D_T$ which preserves the order of levels, i.e.\ for $i < j$, $\lvl_{D_T}(d_i) \leq \lvl_{D_T}(d_j)$. We will show, by induction on $i$, that for $0 \leq i \leq k$ there exists an ordered graph $C_i^\prec \in \mc{J}$ such that:
        \begin{enumerate}
            \item $C_i$ consists of $D_T$ together with, for $1 \leq j \leq i$, an additional $10$ copies of $\scl_{D_T}(d_j)$ freely amalgamated (as graphs) over $\scl_{D_T}(d_j)^\circ$;
            \item for $1 \leq j \leq i$, the $10$ copies of $d_j$ may be labelled as $d_{j, -5}, \cdots, d_{j, -1}, d_{j, 1}, \cdots, d_{j, 5}$ such that $d_{j, -5} \prec \cdots \prec d_{j, -1} \prec d_j \prec d_{j, 1} \prec \cdots \prec d_{j, 5}$ in $\prec_{C_i}$. We will call these the \emph{witness} vertices of $d_j$, and let $W_j$ denote the set of witness vertices of $d_j$.
        \end{enumerate}
        For the base case $i = 0$, take $C_0 = {D_T}^-$. As $C_0 \in \mc{C}_1$ and $\mc{J}$ is a reasonable class of expansions of $(\mc{C}_1, \leq_1)$, there exists a linear order $\prec_{C_0}$ on $C_0$ such that $C_0^\prec \in \mc{J}$, and then $C_0^\prec$ satisfies (1) and (2) vacuously.
			
        For the induction step, assume we have $C_i^\prec \in \mc{J}$ satisfying (1) and (2). Let \[X = \Lvl_0(D_T) \cup \bigcup_{1 \leq j \leq i} \scl_{D_T}(d_j) \cup \bigcup_{1 \leq j \leq i} W_j.\] There is an acyclic $2$-orientation $\tau_i$ of $C_i$ in which $X$ is successor-closed: take the orientation of $D_T$, and orient the two edges of each witness vertex $d_{j, m}$ outwards from $d_{j, m}$. Thus $X \leq_1 C_i$. Note that for $j' > i \geq j$ we have $\lvl_{D_T}(d_{j'}) \geq \lvl_{D_T}(d_{j})$, so $d_{j'} \notin X$ for $j' > i$.
			
        Let $(E, \tau)$ be the free amalgam of $(C_i, \tau_i)$ $11$ times over $(X, \tau_i)$. As $\mc{D}_1$ is a free amalgamation class, we have $(E, \tau) \in \mc{D}_1$. Hence $E \in \mc{C}_1$, and we have $X \leq E$.
			
        Let $\prec_X \,=\, \prec_{C_i}\!|_X$. We have that $X^\prec \in \mc{J}$, so let $\theta_X : X^\prec \to (M_1, \prec_\beta)$ be a $\leq_1$-ordered graph embedding. By the extension property of $M_1$, we have a $\leq_1$-graph embedding $\theta : E \to M_1$ extending $\theta_X$. Define a linear order $\prec_\zeta$ on $E$ by $x \prec_\zeta y$ iff $\theta(x) \prec_\beta \theta(y)$. We have that $\prec_\zeta$ is a linear order on $E$ extending $\prec_X$ on $X$, and that $\theta : (E, \prec_\zeta) \to (M_1, \prec_\beta)$ is a $\leq_1$-ordered graph embedding.
			
        We may label the 11 copies of $C_i$ in $E$ as $C_{i, m}$ ($-5 \leq m \leq 5$), with $\leq_1$-embeddings $\eta_m : C_i \to C_{i, m} \leq E$, and the corresponding copies of $d_{i+1}$ as $d_{i+1, m} \in C_{i, m}$, such that $d_{i+1, -5} \prec \cdots \prec d_{i+1, 5}$ in $\prec_\zeta$. Let ${C_{i+1}}' = C_{i, 0} \cup \{d_{i+1, m} : -5 \leq m \leq 5\}$. We have that $({C_{i+1}}', \tau) \sq_s (E, \tau)$, so ${C_{i+1}}' \leq E$. So $\theta : ({C_{i+1}}', \prec_\zeta) \to (M_1, \prec_\beta)$ is a $\leq_1$-ordered graph embedding.
			
        We have that ${C_{i+1}}'$ consists of a copy $C_{i, 0} = \eta_0(C_i)$ of $C_i$, where $\eta_0|_X = \id_X$ and $\eta_0|_X : (X, \prec_X) \to (X, \prec_\zeta)$ is order-preserving, together with witness vertices $d_{i+1, m}$ (where $1 \leq |m| \leq 5$) for $d_{i+1, 0} = \eta(d_{i+1})$.
			
        Recall that $C_i$ consists of $D_T$ together with, for $1 \leq j \leq i$, the witness vertices for $d_j$, and also that $X$ consists of $\Lvl_0(D_T)$ together with, for $1 \leq j \leq i$, $d_j$ and its witness vertices.
			
        Therefore $({C_{i+1}}', \prec_\zeta)$ consists of a graph-isomorphic copy $\eta_0(D_T)$ of $D_T$, together with witness vertices in $\prec_\zeta$ for $\eta_0(d_1) = d_1, \cdots, \eta_0(d_i) = d_i$ and witness vertices in $\prec_\zeta$ for an additional vertex $\eta_0(d_{i+1})$. We can therefore construct an ordered graph $C_{i+1}^\prec$ isomorphic to $({C_{i+1}}', \prec_\zeta) \in \mc{J}$ such that $C_{i+1}$ consists of $D_T$ together with witness vertices for $d_j$, $1 \leq j \leq i+1$. This completes the induction step. We then let $C_T^\prec = C_k^\prec$.	
    \end{proof}
	
\subsection{\texorpdfstring{$(\mc{J}, \leq_1)$}{J with 1-closed substructures} does not have WAP}
		
    \begin{prop}
        The class $(\mc{J}, \leq_1)$ does not have the weak amalgamation property.
    \end{prop}
    \begin{proof}
        Suppose $(\mc{J}, \leq_1)$ has WAP, seeking a contradiction. Let $\{a_0\} \in \mc{J}$ be a singleton with the trivial linear order. Then there exists $\{a_0\} \leq_1 A^\prec \in \mc{J}$ with $A^\prec$ witnessing WAP for $\{a_0\}$. Take $A \leq_1 B \in \mc{C}_1$ witnessing for $A$ the expansion property of $\mc{J}$ over $(\mc{C}_1, \leq_1)$. (Here we use \Cref{miniffexp}, recalling that $Y = X(\mc{J})$ is a minimal $G$-flow.)
			
        Take $B^+ \in \mc{D}_1$ such that the undirected reduct of $B^+$ is $B$. For each $v \in B^+$ with $\dplus(v) = 1$, add to $B^+$ a new vertex $v'$ and out-edge $vv'$, and call the resulting digraph $C \in \mc{D}_1$. Note that each vertex of $C$ has out-degree $0$ or $2$. We have that $B \leq_1 C^-$ as undirected graphs. Let $q$ be the maximum number of levels in any acyclic reorientation of $C$ (i.e.\ if $C$ when reoriented has levels $0, \cdots, n$, then $q = n + 1$).
			
        For $i = 0, 1$, let $C_i^\prec \in \mc{J}$ be $T_i(q)$-witness ordered graphs for $C$, using Lemma \ref{existsWOG}, and let $D_i, Z_i$ denote $D_{T_i(q)}, Z_{T_i(q)}$ (the notation here is introduced just above Definition \ref{defnWOG}). 
			
        As $B \leq C_i$ witnesses the expansion property for $A$, there exist $\leq_1$-ordered graph embeddings $\zeta_i : A^\prec \to (B, \prec_{C_i}) \leq C_i^\prec$ ($i = 0, 1$). As $A^\prec$ witnesses WAP for $\{a_0\}$, there exists $D \leq M_1$ and $\leq_1$-ordered graph embeddings $\theta_i : C_i^\prec \to (D, \prec_\beta)$ with $\theta_0 \zeta_0 (a) = \theta_1 \zeta_1 (a)$. By Lemma \ref{digraphWOG}, $\theta_i|_{Z_i} : Z_i \to (D, \rho)$ are also digraph embeddings.
			
        If $r$ is a vertex of $C$ of out-degree $0$ in $C$, then as $\theta_i|_{Z_i}$ is a digraph embedding, $\theta_i(r)$ has out-degree $0$ in $\theta_i(C)$. Also $\theta_i$ is a graph embedding, so preserves the sum of out-degrees, and as each vertex of $C$ has out-degree $2$ or $0$ and $\theta_i(C)$ is $2$-oriented, we have that the vertices of $\theta_i(C)$ of out-degree $< 2$ are exactly the $\theta_i(r)$ for $r$ a vertex of $C$ of out-degree $0$.
			
        Let $d = \theta_i\zeta_i(a)$, and let $U_n$ be the set of vertices of $(M_1, \rho)$ that can be reached from $d$ by an outward-directed path of length $\leq n$. As the only vertices of $\theta_i(D_i)$ of out-degree less than $2$ are the leaves of the copies of $T_i$, we have $U_{2q + 1} \sub \theta_i(D_i)$ ($i = 0, 1$).
			
        We now obtain a contradiction by comparing the two cases $i = 0$ and $i = 1$. As $U_{2q + 1} \sub \theta_0(D_0)$, we have that $U_{2q + 1} - U_{q - 1}$ does not contain any (undirected) cycles. But as $U_{2q + 1} \sub \theta_1(D_1)$, we have that $U_{2q + 1} - U_{q - 1}$ contains a $4$-cycle - contradiction.
    \end{proof}
	
    This completes the proof of Theorem \ref{M1meagre}.

\subsection{\texorpdfstring{$\mc{LO}(M_1)$}{LO(M1)} is not minimal}
	
    We now quickly show that $\mc{LO}(M_1)$ is not in fact minimal itself.
	
    \begin{prop} \label{LOM1notmin}
        $\mc{LO}(M_1)$ is not a minimal flow.
    \end{prop}
    \begin{proof}
        Let $\mc{Q}_1$ be the class of ordered graphs $A^\prec$ where $A \in \mc{C}_1$ and $\prec_A$ induces a $2$-orientation $\tau_A$ on $A$: that is $\tau_A = \{ (x, y) \in E_A : y \prec_A x\}$ is a $2$-orientation (which must necessarily be acyclic, as $\prec_A$ is a linear order).
			
        We will show that $\mc{Q}_1$ is a reasonable class of expansions of $(\mc{C}_1, \leq_1)$ (see \cite[Definition 2.14]{EHN19}). Parts (2) and (3) of reasonableness are immediate. For parts (1) and (4), take $A^\prec \in \mc{Q}_1$ and $B \in \mc{C}_1$ with $A \leq_1 B$ (where we allow $A^\prec = \varnothing$). Let $\tau_A$ be the acyclic $2$-orientation induced by $\prec_A$ on $A$. As $A \leq_1 B$, there exists an acyclic $2$-orientation $\tau_B$ of $B$ extending $\tau_A$. Let $\prec_0 = \{(b, b') \in B^2 : b \neq b'$ and there exists an out-path from $b'$ to $b$ in $\tau_B\}$. Then $\prec_0$ is a strict partial order on $B$. $\prec_A$ and $\prec_0$ are compatible, and so we may extend the partial order $\prec_A \cup \prec_0$ arbitrarily to a linear order $\prec_B$ on $B$. Then $\prec_B$ induces $\tau_B$, so $(B, \prec_B) \in \mc{Q}_1$.
        
        By \cite[Theorem 2.15]{EHN19}, we therefore have that $X(\mc{Q}_1)$ is a subflow of $\mc{LO}(M_1)$. To see that it is a proper subflow, we produce a linear order on $M_1$ which does not induce an acyclic $2$-orientation. Let $\prec$ be the linear order of the \Fr limit of the order expansion $(\mc{C}_1^\prec, \leq_1)$ of $(\mc{C}_1, \leq_1)$. By genericity, there exists a graph $A \leq_1 M_1$ consisting of vertices $a, b_1, \cdots, b_3$ and edges $ab_i$ with $b_i \prec a$ ($1 \leq i \leq 3$), so $\prec$ does not induce a 2-orientation.
    \end{proof}

    See \cite[Section 4.6]{Sul22} for an explicit example of a minimal subflow of $\mc{LO}(M_1)$.

\medskip

\textbf{Acknowledgements.} The author would like to thank David Evans for his supervision during this project, which formed part of the second half of the author's PhD thesis.
  
\bibliographystyle{abbrv}
\bibliography{super}

\end{document}